\documentclass[12pt]{amsart}
\usepackage{amssymb,verbatim,enumerate,ifthen}
\usepackage[mathscr]{eucal}
\usepackage[utf8]{inputenc}
\usepackage[T1]{fontenc}
\newtheorem{thm}{Theorem}[section]
\newtheorem{cor}[thm]{Corollary}
\newtheorem{prop}[thm]{Proposition}
\newtheorem{lem}[thm]{Lemma}

\theoremstyle{definition}

\numberwithin{equation}{section}
\def\eq#1{{\rm(\ref{#1})}}
\def\Eq#1#2{\ifthenelse{\equal{#1}{*}}
  {\begin{equation*}\begin{aligned}[]#2\end{aligned}\end{equation*}}
  {\begin{equation}\begin{aligned}[]\label{#1}#2\end{aligned}\end{equation}}}

\def\A{\mathscr{A}}
\def\B{\mathscr{B}}
\def\D{\mathscr{D}}
\def\G{\mathscr{G}}
\def\M{\mathscr{M}}
\def\P{\mathscr{P}}
\newcommand\R{\mathbb{R}}
\newcommand\N{\mathbb{N}}

\newcommand\Mb{\mathbf{M}}

\newcommand\Hc{\mathscr{H}}
\newcommand{\norm}[1]{\left\| #1 \right\| }

\newcommand{\floor}[1]{\left\lfloor #1 \right\rfloor}
\newcommand{\mgp}{M_\otimes}

\title[Characterization of the Hardy property of means]
{Characterization of the Hardy property of means \\ and the best Hardy constants}

\author{Zsolt P\'ales}
\address{Institute of Mathematics, University of Debrecen, Pf.\ 12, 4010 Debrecen, Hungary}
\email{pales@science.unideb.hu}

\author{Pawe\l{} Pasteczka}
\address{Institute of Mathematics, Pedagogical University of Cracow,  Podchor\k{a}\.{z}ych str 2, 30-084 Cracow, Poland}
\email{ppasteczka@up.krakow.pl}

\thanks{The research of the first author was supported by the Hungarian Scientific Research Fund (OTKA) 
Grant K-111651.}

\begin{document}

\begin{abstract} 
The aim of this paper is to characterize in broad classes of means the so-called Hardy means, i.e., those means 
$M\colon\bigcup_{n=1}^\infty \R_+^n\to\R_+$ that satisfy the inequality
\Eq{*}{
  \sum_{n=1}^\infty M(x_1,\dots,x_n) \le C\sum_{n=1}^\infty x_n
}
for all positive sequences $(x_n)$ with some finite positive constant $C$. One of the main results offers a 
characterization of Hardy means in the class of symmetric, increasing, Jensen concave and repetition invariant means 
and also a formula for the best constant $C$ satisfying the above inequality. 
\end{abstract}

\maketitle

\section{Introduction}

Hardy's celebrated inequality (cf.\ \cite{Har25a}, \cite{HarLitPol34}) states that, for $p>1$,
\Eq{1i}{
  \sum_{n=1}^\infty \Big(\frac{x_1+\cdots+x_n}{n}\Big)^p
  \le \Big(\frac{p}{p-1}\Big)^p\sum_{n=1}^\infty x_n^p,
}
for all nonnegative sequences $(x_n)$.

This inequality, in integral form was stated and proved in \cite{Har25a} but it was also pointed out that this discrete 
form follows from the integral version. Hardy's original motivation was to get a simple proof of Hilbert's celebrated 
inequality. About the enormous literature concerning the history, generalizations and extensions of this inequality,
we recommend four recent books \cite{KufMalPer07}, \cite{KufPer00}, \cite{MitPecFin91}, and \cite{OpiKuf90} for the 
interested readers.

In this paper, we follow the approach in generalizing Hardy's
inequality of the paper \cite{PalPer04}. The main idea is to rewrite \eq{1i} in terms of means.

First, replacing $x_n$ by $x_n^{1/p}$ and $p$ by $1/p$, we get that
\Eq{2i}{
  \sum_{n=1}^\infty \Big(\frac{x_1^p+\cdots+x_n^p}{n}\Big)^{1/p}
  \le \Bigl(\frac{1}{1-p}\Bigr)^{1/p}\sum_{n=1}^\infty x_n
}
for $0<p<1$. This inequality was also established for $p<0$ by Knopp \cite{Kno28}. Taking the limit $p\to0$, the 
so-called Carleman inequality (cf.\ \cite{Car32}) can also be derived:
\Eq{3i}{
  \sum_{n=1}^\infty \sqrt[n]{x_1\cdots x_n}
  \le e\sum_{n=1}^\infty x_n.
}
It is also important to note that the constants of the right hand sides of the above inequalities are the smallest 
possible. For further developments and historical remarks concerning inequality \eq{3i}, we refer to the paper 
Pečarić--Stolarsky \cite{PecSto01}.

Now define for $p\in\R$ the $p$th power (or H\"older) mean of the positive numbers $x_1,\dots,x_n$ by
\Eq{PM}{
  \P_p(x_1,\dots,x_n)
   :=\left\{\begin{array}{ll}
    \Big(\dfrac{x_1^p+\cdots+x_n^p}{n}\Big)^{\frac{1}{p}} 
      &\mbox{if }p\neq0, \\[3mm]
      \sqrt[n]{x_1\cdots x_n}\qquad
      &\mbox{if }p=0.
    \end{array}\right.
}
The power mean $\P_1$ is the arithmetic mean which will also be denoted by $\A$ in the sequel. 

Observe that all of the above inequalities are particular cases of the following one
\Eq{4i}{
  \sum_{n=1}^\infty M(x_1,\dots,x_n) \le C\sum_{n=1}^\infty x_n,
}
where $M$ is a mean on $\R_+$, that is, $M$ is a real valued function defined on the set 
$\bigcup_{n=1}^\infty \R_+^{n}$ such that, for all $n\in\N$, $x_1,\dots,x_n>0$,
\Eq{*}{
  \min(x_1,\dots,x_n)\le M(x_1,\dots,x_n)\le\max(x_1,\dots,x_n).
}
In the sequel, a mean $M$ will be called a \emph{Hardy mean} if there exists a positive real constant $C$ such 
that \eq{4i} holds for all positive sequences $x=(x_n)$. The smallest possible extended real value $C$ such that 
\eq{4i} is valid will be called the \emph{Hardy constant of $M$} and denoted by $\Hc_\infty(M)$. Due to the Hardy, 
Carleman, and Knopp inequalities, the $p$th power mean is a Hardy mean if $p<1$. One can easily see that the arithmetic 
mean is not a Hardy mean, therefore the following result holds.

\begin{thm}\label{thm:PM}
Let $p\in\R$. Then, the power mean $\P_p$ is a Hardy mean if and only if $p<1$. In addition, for $p<1$, 
\Eq{*}{
  \Hc_\infty(\P_p)
    =\left\{\begin{array}{ll}
      \big(1-p\big)^{-\frac1p} &\mbox{if }p\neq0, \\[1mm]
      e &\mbox{if }p=0.
    \end{array}\right.
}
\end{thm}

The notion of power means is generalized by the concept of \emph{quasi-arithmetic means} (cf.\ \cite{HarLitPol34}): 
If $I\subseteq\R$ is an interval and $f:I\to\R$ is a continuous strictly monotonic function then the quasi-arithmetic 
mean $\M_f:\bigcup_{n=1}^\infty\R_+^{n}\to\R$ is defined by
\Eq{QM}{
  \M_f(x_1,\dots,x_n)
   :=f^{-1}\bigg(\frac{f(x_1)+\cdots+f(x_n)}{n}\bigg),
   \qquad x_1,\dots,x_n\in I.
}
By taking $f$ as a power function or a logarithmic function on $I=\R_+$, the resulting quasi-arithmetic mean is a power 
mean. It is well-known that \emph{Hölder means are the only homogeneous quasi-arithmetic means} (cf.\ 
\cite{HarLitPol34}, \cite{Pal00a}, \cite{Pas15a}). 

The following result which completely characterizes the Hardy means among quasi-arithmetic means is due to Mulholland 
\cite{Mul32}.

\begin{thm}\label{thm:QM}
Let $f:\R_+\to\R$ be a continuous strictly monotonic function. Then, the quasi-arithmetic mean 
$\M_f$ is a Hardy mean if and only if there exist constants $p<1$ and $C>0$ such that, for all $n\in\N$ and 
$x_1,\dots,x_n>0$,
\Eq{*}{
 \M_f(x_1,\dots,x_n)\le C\P_p(x_1,\dots,x_n).
}
\end{thm}

In 1938 Gini introduced another extension of power means: For $p,q\in\R$, the \emph{Gini mean} $\G_{p,q}$ of 
the variables $x_1,\dots,x_n>0$ is defined as follows:
\Eq{GM}{
  \G_{p,q}(x_1,\dots,x_n)
   :=\left\{\begin{array}{ll}
    \left(\dfrac{x_1^p+\cdots+x_n^p}
           {x_1^q+\cdots+x_n^q}\right)^{\frac{1}{p-q}} 
      &\mbox{if }p\neq q, \\[4mm]
     \exp\left(\dfrac{x_1^p\ln(x_1)+\cdots+x_n^p\ln(x_n)}
           {x_1^p+\cdots+x_n^p}\right) \quad
      &\mbox{if }p=q.
    \end{array}\right.
}
Clearly, in the particular case $q=0$, the mean $\G_{p,q}$ reduces to the $p$th Hölder mean $\P_p$. It is also obvious 
that $\G_{p,q}=\G_{q,p}$. A common generalization of quasi-arithmetic means and Gini means can be obtained in terms of 
two arbitrary real functions. These means were introduced by Bajraktarević \cite{Baj58}, \cite{Baj69} in 1958. Let 
$I\subseteq\R$ be an interval and let $f,g:I\to\R$ be continuous functions such that $g$ is positive and $f/g$ is 
strictly monotone. Define the \emph{Bajraktarević mean} $\B_{f,g}:\bigcup_{n=1}^\infty I^{n}\to\R$ by
\Eq{*}{
  \B_{f,g}(x_1,\dots,x_n)
   :=\Big(\frac{f}{g}\Big)^{-1}\bigg(\frac{f(x_1)+\cdots+f(x_n)}
                      {g(x_1)+\cdots+g(x_n)}\bigg),
                      \qquad x_1,\dots,x_n\in I.
}
One can check that $\B_{f,g}$ is a mean on $I$. In the particular case $g\equiv1$, the mean $\B_{f,g}$ reduces to 
$\M_f$, that is, the class of Bajraktarević means is more general than that of the quasi-arithmetic means. By taking 
power functions, we can see that the Gini means also belong to this class. It is a remarkable result of Aczél and 
Daróczy \cite{AczDar63c} that the homogeneous means among the Bajraktarević means defined on $I=\R_+$ are exactly the 
Gini means.

Finally, we recall the concept of the most general means considered in this paper, the notion of the \emph{deviation 
means} introduced by Daróczy \cite{Dar72b} in 1972. A function $E:I\times I\to\R$ is called a \emph{deviation 
function} on $I$ if $E(x,x)=0$ for all $x\in I$ and the function $y\mapsto E(x,y)$ is continuous and strictly decreasing
on $I$ for each fixed $x\in I$. The \emph{$E$-deviation mean} or \emph{Daróczy mean} of some values $x_1,\dots,x_n\in 
I$ is now defined as the unique solution $y$ of the equation 
\Eq{*}{
  E(x_1,y)+\cdots+E(x_n,y)=0
}
and is denoted by $\D_E(x_1,\dots,x_n)$. It is immediate to see that the arithmetic deviation $A(x,y)=x-y$ generates 
the arithmetic mean. More generally, if $E:I\times I\to\R$ is of the form $E(x,y):=f(x)-g(x)\big(\frac{f}{g}\big)(y)$ 
for some continuous function $f,g:I\to\R$ such that $g$ is positive and $f/g$ is strictly monotone then $\D_E=\B_{f,g}$.
Thus, Hölder means, quasi-arithmetic means, Gini means and Bajraktarević means are particular Daróczy means. The class 
of deviation means was slightly generalized to the class of quasi-deviation means and this class was completely 
characterized by Páles in \cite{Pal82a}.

The following result, which gives necessary and also sufficient conditions for the Hardy property of deviation means was 
established by Páles and Persson \cite{PalPer04} in 2004.

\begin{thm}
Let $E:\R_+\times\R_+\to\R$ be a deviation on $\R_+$. If $\D_E$ is a Hardy mean, then there exists a positive constant 
$C$ such that
\Eq{*}{
  \D_E(x_1,\dots,x_n)\le C\A(x_1,\dots,x_n)
}
holds for all $n\in\N$ and $x_1,\dots,x_n>0$ and there is no positive constant $C^*$ such that 
\Eq{*}{
   C^*\A(x_1,\dots,x_n)\le \D_E(x_1,\dots,x_n)
}
be valid on the same domain. Conversely, if 
\Eq{*}{
  \D_E(x_1,\dots,x_n)\le C\P_p(x_1,\dots,x_n)
}
is satisfied with a parameter $p<1$ and a positive constant $C$, then $\D_E$ is a Hardy mean.
\end{thm}

As a corollary of the previous result, necessary and also sufficient conditions for the Hardy property were established 
in the class of Gini means by Páles and Persson \cite{PalPer04} in 2004.

\begin{thm}\label{thm:PalPer04}
Let $p,q\in\R$. If $\G_{p,q}$ is a Hardy mean, then 
\Eq{*}{
  \min(p,q)\le0 \qquad\mbox{and}\qquad \max(p,q)\le1.
}
Conversely, if 
\Eq{*}{
  \min(p,q)\le0 \qquad\mbox{and}\qquad \max(p,q)<1,
}
then $\G_{p,q}$ is a Hardy mean.
\end{thm}

It has been an open problem since 2004 whether the second condition was a necessary and sufficient condition for the 
Hardy property and also the best Hardy constant was to be determined.

The necessary and sufficient condition for the Hardy property of Gini means was finally found by Pasteczka 
\cite{Pas15c} in 2015. The key was the following general necessary condition for the Hardy property.

\begin{lem}\label{lem:Pas15}
Assume that $M\colon\bigcup_{n=1}^\infty \R_+^n\to\R_+$ is a Hardy mean. Then, for all positive non-$\ell_1$ sequences
$(x_n)$,
\Eq{*}{
   \liminf_{n\to\infty} x_n^{-1}M(x_1,\dots,x_n)<\infty.
}
\end{lem}

Applying this necessary condition in the class of Gini means with the harmonic sequence $x_n:=\frac1n$, Pasteczka 
\cite{Pas15c} obtained the following characterization of the Hardy property for Gini means. 

\begin{thm}\label{thm:Pas15}
Let $p,q\in\R$. Then $\G_{p,q}$ is a Hardy mean if and only if 
\Eq{*}{
  \min(p,q)\le0 \qquad\mbox{and}\qquad \max(p,q)<1.
}
\end{thm}

There was no progress, however, in establishing the Hardy constant of the Gini means. There was only an upper estimate 
obtained by Páles and Persson in \cite{PalPer04}.

Motivated by all these preliminaries, the purpose of this paper is twofold: \\
--- To find (in terms of easy-to-check properties) a large subclass of Hardy means. \\
--- To obtain a formula for the Hardy constant in that subclass of means.

\section{Means and their basic properties}

For investigating the Hardy property of means, we recall several relevant notions. Let $I\subseteq\R$ be an 
interval and let $M \colon\bigcup_{n=1}^{\infty} I^n \to I$ be an arbitrary mean.

We say that $M$ is \emph{symmetric}, \emph{(strictly) increasing}, and \emph{Jensen convex (concave)} if, for all 
$n\in\N$, the $n$-variable restriction $M|_{I^n}$ is a symmetric, (strictly) increasing in each of its variables, and 
Jensen convex (concave) on $I^n$, respectively. If $I=\R_+$, we can analogously define the notion of homogeneity of $M$.

The mean $M$ is called \emph{repetition invariant} if, for all $n,m\in\N$
and $(x_1,\dots,x_n)\in I^n$, the following identity is satisfied
\Eq{*}{
  M(\underbrace{x_1,\dots,x_1}_{m\text{-times}},\dots,\underbrace{x_n,\dots,x_n}_{m\text{-times}})
   =M(x_1,\dots,x_n).
}

The mean $M$ is \emph{strict} if for any $n\geq2$
and any non-nonstant vector $(x_1,\dots,x_n)\in I^n$,
\Eq{*}{
\min(x_1,\dots,x_n)< M(x_1,\dots,x_n)< \max(x_1,\dots,x_n).
}

The mean $M$ is said to be \emph{$\min$-diminishing} if, for any $n\geq2$
and any non-nonstant vector $(x_1,\dots,x_n)\in I^n$,
\Eq{*}{
   M(x_1,\dots,x_n,\min(x_1,\dots,x_n)) < M(x_1,\dots,x_n).
}
It is easy to check that quasi-arithmetic means are symmetric, strictly increasing, repetition invariant, strict and 
$\min$-diminishing. More generally, deviation means are symmetric, repetition invariant, strict and $\min$-diminishing 
(cf.\ \cite{Pal82a}). The increasingness of a deviation mean $\D_E$ is equivalent to the increasingness of the 
deviation $E$ in its first variable. The Jensen concavity/convexity of quasi-arithmetic and also of deviation means can 
be characterized by the concavity/convexity conditions on the generating functions. All these characterizations are 
consequences of the general results obtained in a series of papers by Losonczi 
\cite{Los70a,Los71a,Los71b,Los71c,Los73a,Los77} (for Bajraktarević means) and by Daróczy 
\cite{DarLos70,Dar71b,Dar72b,DarPal82,DarPal83} and Páles 
\cite{Pal82b,Pal83a,Pal83b,Pal83c,Pal84a,Pal85a,Pal87d,Pal88a,Pal88d,Pal88e} (for (quasi-)deviation means).

\subsection{Kedlaya means}
The notion of a Kedlaya mean that we introduce below turns out to be indispensable for the investigation of 
Hardy means. A mean $M \colon \bigcup_{n=1}^{\infty} I^n \to I$ is called a \emph{Kedlaya mean} if, for all $n\in\N$ and
$(x_1,\ldots,x_n)\in I^n$,
\Eq{KI}{
\frac{M(x_1)+M(x_1,x_2)+\cdots+M(x_1,\ldots,x_n)}n 
  \le M\left(x_1,\frac{x_1+x_2}2,\ldots,\frac{x_1+\cdots+x_n}n \right).
}
The motivation for the above terminology comes from the papers \cite{Ked94,Ked99} by Kedlaya, where he proved that the 
geometric mean satisfies the inequality \eq{KI}, i.e., it is a Kedlaya mean. The next result provides a sufficient 
condition in order that a mean be a Kedlaya mean. 

\begin{thm}\label{thm:Kedlaya_type}
Every symmetric, Jensen concave and repetition invariant mean is a Kedlaya mean.
\end{thm}

\begin{proof}
Let $M \colon \bigcup_{n=1}^{\infty} I^n \to I$ be a symmetric, Jensen concave and repetition invariant mean.
Fix $n \in \N$ and $(x_1,\ldots,x_n) \in I^n$. Adopting Kedlaya's original proof, for $(i,j,k)\in\{1,\dots,n\}^3$, 
we define
\Eq{*}{
a_k(i,j):=(n-1)! \cdot \binom{n-i}{j-k} \binom{i-1}{k-1} \bigg/ \binom{n-1}{j-1}
  =\frac{(n-i)!(n-j)!(i-1)!(j-1)!}{(n-i-j+k)!(i-k)!(j-k)!(k-1)!}.
}
To provide the corectness of this definition we assume that $m!=\infty$ for negative integers $m$ (it is a natural 
extension of gamma function). Then, according to \cite{Ked94}, we have the following properties:
\begin{enumerate}[\quad(1)\quad]
\item $a_k(i,j) \ge 0$ for all $i,\,j,\,k$;
\item $a_k(i,j) \in \N$ for all $i,\,j,\,k$;
\item $a_k(i,j) = 0$ for $k> \min(i,j)$;
\item $a_k(i,j) = a_k(j,i)$ for all $i,\,j,\,k$;
\item $\sum_{k=1}^{n}a_k(i,j)=(n-1)!$ for all $i,\,j$;
\item $\sum_{i=1}^{n}a_k(i,j)=\begin{cases} n!/j & \text{ for }k \le j, \\ 0 & \text{ for } k>j. \end{cases}$
\end{enumerate}

Let us construct a matrix $A$ of size $n! \times n!$ divided into $n^2$ blocks $(A_{i,j})_{i,j \in\{1,\dots, n\}}$ of 
size $(n-1)! \times (n-1)!$.

The first row of each block $A_{i,j}$ contains the number $k$ exactly $a_k(i,j)$ times for $k \in \{1,\ldots,n\}$; this 
could be done by (5). The subsequent rows are all cyclic permutations of the first one. In this way each row and each 
column of $A_{i,j}$ contains the number $k$ exactly $a_k(i,j)$ times.

Now, let $c_p(k)$ denote the occurrence of the number $k$ appearing in the $p$th row of $A$. Then, by (4),
$c_p(k)$ is equal to the number of occurrences of $k$ in the $p$th column of $A$.

We are going to calculate $c_p(k)$. The $p$th row has a nonempty intersection with the block $A_{i,j}$ if
\Eq{*}{
  i=\floor{\frac{p-1}{(n-1)!}}+1=:b(p).
}
Whence, applying property (6), we get
\Eq{*}{
c_p(k)=\sum_{i=1}^n a_k(i,b(p))
  =\begin{cases} n!/b(p) & \text{ for }k \le b(p), \\[2mm] 0 & \text{ for } k>b(p).\end{cases}
}
Now, let us consider the matrix $A'$ obtained from $A$ by replacing $k \mapsto x_k$ for $k \in \{1,\ldots,n\}$.
We will calculate the mean value of the elements of $A'$ in two different ways. First, we calculate the mean $M$ of 
each column of $A'$. By the Jensen concavity of $M$, the arithmetic mean of the results so obtained does not exceed the 
result of calculating arithmetic mean of each row of $A'$ and then taking the $M$ mean of the resulting vector of 
length 
$n!$. Whence, using the symmetry and the repetition invariance of $M$, we obtain
\Eq{*}{
\frac{1}{n!} \Big( (n-1)! M(x_1)&+(n-1)! M(x_1,x_2)+\cdots+(n-1)! M(x_1,x_2,\ldots,x_n)\Big)\\
&\le M\Big(x_1,\frac{x_1+x_2}{2},\ldots,\frac{x_1+x_2+\cdots+x_n}{n} \Big)
}
which simplifies to the inequality \eq{KI} to be proved.
\end{proof}

\begin{cor}\label{cor:kedlaya_type}
If, in addition to the assumptions of Theorem~\ref{thm:Kedlaya_type}, $M$ is also increasing and $I=\R_{+}$, then
\Eq{*}{
M(x_1)&+M(x_1,x_2)+\cdots+M(x_1,\ldots,x_n) \\
&\le n \cdot M\left(x_1+\cdots+x_n,\frac{x_1+\cdots+x_n}2,\ldots,\frac{x_1+\cdots+x_n}n \right).
}

\end{cor}

\subsection{Gaussian product}
The Gaussian product of means is a broad extension of Gauss' idea of the arithmetic-geometric mean. 
In 1800 (this year is due to \cite{ToaToa05}) he proposed the following two-term recursion:
\Eq{*}{
x_{n+1}=\frac{x_n+y_n}2, \quad y_{n+1}=\sqrt{x_ny_n}, \qquad n=0,1,\dots,
}
where $x_0$ and $y_0$ are positive numbers. Gauss \cite[p.\ 370]{Gau18} proved that both $(x_n)_{n=1}^\infty$ and 
$(y_n)_{n=1}^\infty$ converge to a common limit, which is called arithmetic-geometric mean of the initial values $x_0$ 
and $y_0$. J.~M.~Borwein and P.~B.~Borwein \cite{BorBor87} extended some earlier ideas \cite{FosPhi84a,Leh71,Sch82} and 
generalized this iteration to a vector of continuous, strict means of an arbitrary length.
For several recent results about Gaussian product of means see the papers by Baják and Páles 
\cite{BajPal09b,BajPal09a,BajPal10,BajPal13}, by Daróczy and Páles \cite{Dar05a,DarPal02c,DarPal03a}, 
by Głazowska \cite{Gla11b,Gla11a}, by Matkowski \cite{Mat99b,Mat02b,Mat05,Mat13}, and by Matkowski and Páles 
\cite{MatPal15}.

Given $N \in \N$ and a vector $(M_1,\dots,M_N)$ of means defined on a common interval $I$ and having values in $I$ 
(i.e. $M_i\colon \bigcup_{n=1}^\infty I^n \to I$ for every $i \in\{1,\dots,N\}$), let us introduce the mapping 
$\mathbf{M} \colon \bigcup_{n=1}^\infty I^n \to I^N$ by
\Eq{*}{
\mathbf{M}(v):=(M_1(v),M_2(v),\dots,M_N(v)),\qquad v \in \bigcup_{n=1}^\infty I^n.
}
Whenever, for every $i \in \{1,\dots,N\}$ and every $v \in\bigcup_{n=1}^\infty I^n$, the limit of 
iterations $\lim_{k \to \infty} [\mathbf{M}^{k}(v)]_i$ exists and does not depend on $i$, then the value of this limit 
will be called the \emph{Gaussian product} of $(M_1,\dots,M_N)$ evaluated at $v$. We will denote this limit by 
$\mgp(v)$. It is well-known that the Gaussian product can equivalently be defined as a unique function satisfying the 
following two properties:
\begin{enumerate}[\quad(i)\quad]
 \item $\mgp \circ \mathbf{M}(v) = \mgp(v)$ for all $v \in \bigcup_{n=1}^\infty I^n$,
 \item $\min(v) \le \mgp(v) \le \max(v)$ for all $v \in \bigcup_{n=1}^\infty I^n$.
\end{enumerate}

Frequently, whenever each of the means $M_i$, $i \in \{1,\dots,N\}$ has a certain property, then $\mgp$ inherits this 
property. The lemma below (in view of Theorem~\ref{thm:Kedlaya_type}) is its very useful exemplification.

\begin{lem}
\label{lem:Gaussian_concave}
Let $I$ be an interval, $N \in \N$, and let $(M_1,\dots,M_N) \colon \bigcup_{n=1}^{\infty} I^n\to I^N$. If, for each 
$i\in\{1,\dots,N\}$, $M_i$ is symmetric/homogeneous/repetition invariant/increasing and Jensen concave/, then so is 
their Gaussian product $\mgp$.
\end{lem}

\begin{proof}
The first four properties are naturally inherited by all of the functions $[\Mb^k]_i$, for $k \in \N$, 
$i\in\{1,\ldots,N\}$ and, finally, by their pointwise limit. The verification of the statement about the Jensen 
concavity is just a little bit more sophisticated. In fact, the idea presented below could also be adapted to the 
remaining properties.

Assume that $M_1,\dots,M_N$ are increasing and Jensen concave. We will prove that $\mgp$ is Jensen concave. Let 
$x^{(0)}$, $y^{(0)}$ be the equidimensional vectors and $m^{(0)}=\tfrac12 (x^{(0)}+y^{(0)})$. Let
\Eq{*}{
x^{(k+1)}=\Mb(x^{(k)}),\quad y^{(k+1)}=\Mb(y^{(k)}),\quad m^{(k+1)}=\Mb(m^{(k)}),\qquad k \in \N.
}
We are going to prove that 
\Eq{eq:Gauss_ind}{
[m^{(k)}]_i \ge \tfrac12 [x^{(k)}+y^{(k)}]_i\quad\text{ for any }i\in \{ 1,\ldots,N\}\text{ and }k \in 
\N. 
}
Obviously, this holds for $n=0$. Let us assume that \eqref{eq:Gauss_ind} holds for some $n \in \N$ and any $i$.
Then, by the increasingness and Jensen concavity of $M_i$,
\Eq{*}{
[m^{(k+1)}]_i&=M_i(m^{(k)}) \ge M_i \left( \tfrac12 (x^{(k)}+y^{(k)}) \right) \ge \tfrac12  \left( 
M_i(x^{(k)})+M_i(y^{(k)}) \right) \\
&= \tfrac12 \left( [x^{(k+1)}]_i+[y^{(k+1)}]_i \right) = \tfrac12 [x^{(k+1)}+y^{(k+1)}]_i.
}
Upon taking the limit $k \to \infty$, one gets 
\Eq{*}{
\mgp\left(\frac{x^{(0)}+y^{(0)}}2\right)=\mgp(m^{(0)}) \ge \tfrac12 \left(\mgp(x^{(0)})+\mgp(y^{(0)})\right),\
}
which proves that $\mgp$ is Jensen concave, indeed.
\end{proof}

\section{Main Results}

In the sequel, let $I\subseteq\R$ be a nondegenerate interval such that $\inf I=0$. We will denote by 
$\ell_1(I)$ the collection of all sequences $x=(x_n)_{n=1}^\infty$ such that, for all $n\in\N$, $x_n\in I$ and 
$\|x\|_1:=\sum_{n=1}^\infty x_n$ is convergent, i.e., $x\in\ell_1$.

For a given mean $M \colon \bigcup_{n=1}^{\infty} I^n \to I$ let $\Hc_\infty(M)$ be the smallest nonnegative
extended real number, called the \emph{Hardy constant of $M$}, such that
\Eq{HD}{
\sum_{n=1}^\infty M(x_1,\dots,x_n)\le \Hc_\infty(M)\sum_{n=1}^\infty x_n,\qquad 
(x_n)_{n=1}^\infty\in\ell_1(I).
}
If $\Hc_\infty(M)$ is finite, then we say that $M$ is a \emph{Hardy mean}. Given also $n\in\N$, we define $\Hc_n(M)$ to 
be the smallest nonnegative number such that
\Eq{HnD}{
M(x_1)+\dots+M(x_1,\dots,x_n) \le \Hc_n(M)(x_1+\cdots+x_n),\qquad (x_1,\dots,x_n) \in I^n.
}
Due to the mean value property of $M$, for $n\in\N$, we easily obtain that
$1\leq\Hc_n(M)\leq n$. The sequence $\big(\Hc_n(M)\big)_{n=1}^\infty$ will be called the \emph{Hardy sequence of $M$}.

Several estimates of the Hardy sequences for power means were given during the years. For example Kaluza and Szeg\H{o} 
\cite{KalSze27} proved $\Hc_n(\P_p) \le \tfrac{1}{n(\exp(1/n)-1)} \cdot \Hc_\infty(\P_p)$ for $p \in [0,1)$ and $n \in 
\N$. Moreover it is known \cite[p.267]{HarLitPol34} that $\Hc_n(\P_0) \le (1+\tfrac1n)^n$ for all $n \in \N$.

The basic properties of the Hardy sequence are established in the following
%\section{Auxiliary results}
\begin{prop}
\label{lem:1}
For every mean $M \colon \bigcup_{n=1}^{\infty} I^n \to I$, its Hardy sequence is nondecreasing and
\Eq{Hlim}{
  \lim_{n\to\infty}\Hc_n(M)=\Hc_\infty(M).
}
\end{prop}

\begin{proof} To verify the nondecreasingness of the Hardy sequence of $M$, let $(x_1,\dots,x_n) \in I^n$
and $\varepsilon\in I$ be arbitrary. 
Applying inequality \eq{HnD} to the sequence $(x_1,\ldots,x_n,\varepsilon)\in I^{n+1},$ we obtain
\Eq{*}{
M(x_1)+\cdots+M(x_1,\ldots,x_n)
&\le M(x_1)+\cdots+M(x_1,\ldots,x_n)+M(x_1,\ldots,x_n,\varepsilon)\\
&\le \Hc_{n+1}(M) (x_1+\cdots+x_n+\varepsilon).
}
Upon taking the limit $\varepsilon\to0$, it follows that
\Eq{*}{
M(x_1)+\cdots+M(x_1,\ldots,x_n)\leq \Hc_{n+1}(M) (x_1+\cdots+x_n)
}
for all $(x_1,\dots,x_n) \in I^n$. Hence $\Hc_{n}(M)\leq\Hc_{n+1}(M)$.

To prove \eq{Hlim}, we will show first that $\Hc_{n}(M)\leq\Hc_\infty(M)$ for all $n\in\N$.
If $\Hc_\infty(M)=\infty$ then this inequality is obvious, hence we may assume that $M$ is a Hardy mean.
Fix $n \in \N$ and $(x_1,\ldots,x_n)\in I^n$ and choose $\varepsilon\in I$ arbitrarily. Applying  
\eqref{HD} to the sequence 
$(x_1,\ldots,x_n,\tfrac{\varepsilon}{2},\tfrac{\varepsilon}{4},\tfrac{\varepsilon}{8},\ldots)\in\ell_1(I)$, one 
gets
\Eq{*}{
M(x_1)&+\cdots+M(x_1,\ldots,x_n) \\
&\le M(x_1)+\cdots+M(x_1,\ldots,x_n)+M(x_1,\ldots,x_n,\tfrac{\varepsilon}{2})
+M(x_1,\ldots,x_n,\tfrac{\varepsilon}{2},\tfrac{\varepsilon}{4})+\cdots \\
&\le \Hc_\infty(M)(x_1+\cdots+x_n+\tfrac{\varepsilon}{2}+\tfrac{\varepsilon}{4}+\cdots)\\
&= \Hc_\infty(M)(x_1+\cdots+x_n+\varepsilon).
}
Upon passing the limit $\varepsilon \to 0$, we get
\Eq{*}{
M(x_1)+\cdots+M(x_1,\ldots,x_n) \le \Hc_\infty(M)(x_1+\cdots+x_n),
}
which implies $\Hc_n(M)\leq\Hc_\infty(M)$. Using this inequality, we have also proved that in \eq{Hlim} $\leq$ 
holds instead of equality.

To prove the reversed inequality in \eq{Hlim}, let $(x_n)_{n=1}^\infty\in\ell_1(I)$ be arbitrary.
Then, for all $n\leq k$, we have that
\Eq{*}{
M(x_1)+\cdots+M(x_1,\ldots,x_n)\leq \Hc_n(M) (x_1+\cdots+x_n)\leq \Hc_k(M) (x_1+\cdots+x_n)
}
Now taking the limit as $k \to \infty$, we obtain that
\Eq{*}{
M(x_1)+\cdots+M(x_1,\ldots,x_n) \le \lim_{k \to \infty} \Hc_k(M) \cdot (x_1+\cdots+x_n)
}
holds for all $n\in\N$. Finally taking the limit as $n\to\infty$, it follows that $M$ satisfies 
\Eq{*}{
\sum_{n=1}^\infty M(x_1,\dots,x_n)\le \lim_{k \to \infty} \Hc_k(M) \sum_{n=1}^\infty x_n
}
which yields that the reversed inequality in \eq{Hlim} is also true.
\end{proof}

In what follows, we show that the inequality \eq{HD} is strict in a broad class of means. 

\begin{prop}
Let $I\subseteq \R_{+}$ and $M \colon \bigcup_{n=1}^{\infty} I^n \to I$.
If $M$ is a $\min$-diminishing, increasing and repetition invariant Hardy mean, then 
\Eq{*}{
\sum_{n=1}^\infty M(x_1,\dots,x_n)<\Hc_\infty(M) \sum_{n=1}^\infty x_n ,\qquad (x_n)_{n=1}^\infty \in 
\ell_1(I).
}
\end{prop}

\begin{proof}
Let $x=(x_n)_{n=1}^\infty \in \ell_1(I)$ be arbitrary.
If $x_l<x_k$ for some $l<k$ then, for the sequence
\Eq{*}{
  x'_n=\begin{cases} x_n & n \notin \{k,l\}, \\ x_k & n=l, \\ x_l &n=k, \end{cases}
}
we have
\Eq{*}{
M(x_1,\ldots,x_n)&=M(x_1',\ldots,x_n') \qquad\text{ for }n < l \text{ or } n \ge k,\\
M(x_1,\ldots,x_n)&\leq M(x_1',\ldots,x_n') \qquad \text{ for }n \in \{l,\ldots,k-1\}.
}
Therefore
\Eq{*}{
M(x_1)+\cdots+M(x_1,\dots,x_n)+\cdots\leq M(x_1')+\dots+M(x_1',\dots,x_n')+\cdots.
}
Whence we may assume that $x$ is non-increasing.

Let $\hat{x}=(x_1,x_1,\ldots,x_n,x_n,\dots)$. Then, by the repetition invariance and the $\min$-diminishing property
of $M$, we get
\Eq{*}{
M(x_1,\ldots,x_n)&=M(\hat{x}_1,\ldots,\hat{x}_{2n}), \\
M(x_1,\ldots,x_n)&=M(\hat{x}_1,\ldots,\hat{x}_{2n-1}) \qquad\text{ if } x_1 = x_n,\\
M(x_1,\ldots,x_n)&<M(\hat{x}_1,\ldots,\hat{x}_{2n-1}) \qquad\text{ if } x_1 \neq x_n.
}
Since $x_n\to0$ as $n\to\infty$, hence $x_1 \neq x_n$ holds for some $n$. Therefore
\Eq{*}{
2\cdot \sum_{n=1}^{\infty} M(x_1,\ldots,x_n) < \sum_{n=1}^{\infty} M(\hat{x}_1,\ldots,\hat{x}_{n}) 
\le \Hc_\infty(M) \sum_{n=1}^{\infty}\hat{x}_n = 2\Hc_\infty(M) \sum_{n=1}^{\infty}x_n.
}
This completes the proof of the proposition.
\end{proof}

The next result offers a fundamental lower estimate for the Hardy constant of a mean.

\begin{thm}
\label{thm:lower_estim_Hc}
Let $M \colon \bigcup_{n=1}^{\infty} I^n \to I$ be a mean. Then, for all sequences $(x_n)_{n=1}^\infty$ in $I$ that 
does not belong to $\ell_1$,
\Eq{0}{
  \liminf_{n \to \infty} x_n^{-1}M(x_1,\ldots,x_n)\leq\Hc_\infty(M).
}
\end{thm}

\begin{proof} Assume, on the contrary, that 
\Eq{1}{
  \Hc_\infty(M)<\liminf_{n\to\infty} x_n^{-1}M(x_1,\dots,x_n).
}
Then, there exists $\varepsilon>0$ and $n_0$ such that, for all $n\geq n_0$,
\Eq{*}{
  (1+\varepsilon) \Hc_\infty(M)x_n < M(x_1,\dots,x_n).
}
Choose $n_1>n_0$ such that
\Eq{2}{
  \sum_{n=1}^{n_0} x_n \leq \varepsilon\sum_{n=n_0+1}^{n_1} x_n 
}
Thus, using \eq{1}, Proposition \ref{lem:1}, and finally \eq{2}, we obtain
\Eq{*}{
  \sum_{n=n_0+1}^{n_1} (1+\varepsilon)\Hc_\infty(M)x_n
   &<\sum_{n=n_0+1}^{n_1} M(x_1,\dots,x_n)\leq\sum_{n=1}^{n_1} M(x_1,\dots,x_n)\\
   &\leq \Hc_{n_1}(M) \sum_{n=1}^{n_1} x_n  \leq \Hc_{\infty}(M) \sum_{n=1}^{n_1} x_n
   \leq (1+\varepsilon)\Hc_\infty(M) \sum_{n=n_0+1}^{n_1} x_n.
}
This contradiction validates \eq{0}.
\end{proof}

The main result of our paper is contained in the following theorem.

\begin{thm}\label{thm:main}
Let $M \colon \bigcup_{n=1}^{\infty} \R_+^n \to\R_+$ be an increasing, symmetric, repetition invariant, and Jensen 
concave mean. Then
\Eq{HCF}{
\Hc_\infty(M)=\sup_{y>0} \liminf_{n \to \infty} \frac ny \cdot M\left(\frac y1,\frac y2,\ldots,\frac yn 
\right).
}
\end{thm}

As a trivial consequence of the above result, $M$ is a Hardy mean if and only if the number $\Hc_\infty(M)$ given in 
\eq{HCF} is finite. 

\begin{proof} 
For the proof of the theorem, denote 
\Eq{*}{
C:=\sup_{y>0} \liminf_{n \to \infty} \frac ny \cdot M\left(\frac y1,\frac y2,\ldots,\frac yn \right).
}
The inequality $\Hc_\infty(M)\geq C$ is simply a consequence of Theorem~\ref{thm:lower_estim_Hc}.

To show the reversed inequality, we may assume that $C$ is finite. Fix $x \in \ell_1(\R_+)$ and denote $y:=\|x\|_1$. 
Then 
there exists a sequence $(n_k)$, $n_k \to \infty$ such that 
\Eq{*}{
n_k \cdot M\Big(\frac{y}1,\frac{y}2,\ldots,\frac{y}{n_k}\Big) \le (C+\tfrac{1}{k})y,\qquad k \in \N.
}
By the increasingness of $M$ and by the obvious inequality $x_1+\cdots+x_{n_k}\leq y$, the previous inequality 
yields
\Eq{*}{
n_k \cdot M\left(x_1+\cdots+x_{n_k},\frac{x_1+\cdots+x_{n_k}}2,\ldots,\frac{x_1+\cdots+x_{n_k}}{n_k}\right) 
\le \big(C+\tfrac{1}{k}\big)y,\qquad k \in \N.
}
Therefore, in view of Corollary~\ref{cor:kedlaya_type}, we obtain
\Eq{*}{
M(x_1)+M(x_1,x_2)+\cdots+M(x_1,\ldots,x_{n_k}) \le (C+\tfrac{1}{k})y,\qquad k \in \N.
}
Upon passing the limit $k \to \infty$, one gets
\Eq{*}{
  \sum_{n=1}^\infty M(x_1,\dots,x_n) \le Cy=C \norm{x}_1.
}
This completes the proof of inequality $\Hc_\infty(M)\leq C$.
\end{proof}

\begin{cor}
\label{cor:main_hom}
If, in addition to the assumptions of Theorem~\ref{thm:main}, $M$ is also homogeneous, then
\Eq{*}{
  \Hc_\infty(M)= \lim_{n \to \infty} n \cdot M\left(1,\tfrac 12,\ldots,\tfrac 1n \right).
}
\end{cor}

\begin{proof} In view of the previous theorem we only need to prove that the limit of the sequence $(p_n)$ exists 
(possible infinite), where
\Eq{*}{
p_n:=n \cdot M\left(1,\tfrac 12,\ldots,\tfrac 1n \right).
}
For, it suffices to show that this sequence is nondecreasing. 
Fix $n \in \N$. Let us consider the two vectors $u,v$ of dimension $n(n+1)$ defined by
\Eq{*}{
u&:=(\underbrace{n,\ldots,n}_{n+1},\underbrace{\tfrac{n}2,\ldots,\tfrac{n}2}_{n+1},\ldots,\underbrace{\tfrac{n}{n-1},
\ldots,\tfrac{n}{n-1}}_{n+1},\underbrace{1,\ldots,1}_{n+1});\\
v&:=(\underbrace{n+1,\ldots,n+1}_{n}, \underbrace{\tfrac{n+1}2,\ldots,\tfrac{n+1}2}_{n}, \ldots, 
\underbrace{\tfrac{n+1}n,\ldots,\tfrac{n+1}n}_{n},\underbrace{1,\ldots,1}_{n}).
}
By the homogeneity and repetition invariance of $M$, we have that $M(u)=p_n$ and $M(v)=p_{n+1}$. 
Divide vectors $u$ and $v$ into $n+1$ parts of dimension $n$:
\Eq{*}{
u^{(i)}&:=(\underbrace{\tfrac{n}{i},\ldots,\tfrac{n}{i}}_{i},\underbrace{\tfrac{n}{i+1},\ldots,\tfrac{n}{i+1}}_{n-i}),
&&\qquad i=0,\ldots,n;\\
v^{(i)}&:=(\underbrace{\tfrac{n+1}{i+1},\ldots,\tfrac{n+1}{i+1}}_{n}),&&\qquad i=0,\ldots,n.
}
For $i \ge 1$, each element $\tfrac ni$ appears $(n-i+1)$ times in $u^{(i-1)}$ and $i$ times in $u^{(i)}$, that is, 
$(n+1)$ times altogether. Therefore, the arithmetic mean of $u^{(i)}$, denoted by $\A(u^{(i)})$, is equal to 
$\tfrac{n+1}{i+1}$ for $i=1,\ldots,n$ and $\A(u^{(0)})=n$.

Let $u^{(i)}_k$, for $k=1,\ldots,n!$ and $i=0,\ldots,n$, denote the vectors that are obtained from all possible 
permutations of the components of $u^{(i)}$. Observe that 
\Eq{*}{
  (u^{(0)},v^{(1)},\ldots,v^{(n)})=\frac{1}{n!} \sum_{k=1}^{n!} (u^{(0)}_k,\ldots,u^{(n)}_k).
}
Then, by the increasingness, Jensen concavity and symmetry of the mean $M$, we obtain
\Eq{*}{
p_{n+1} = M(v) 
&= M(v^{(0)},v^{(1)},\ldots,v^{(n)})\geq M(u^{(0)},v^{(1)},\ldots,v^{(n)})\\
&\ge \frac{1}{n!} \sum_{k=1}^{n!} M(u^{(0)}_k,\ldots,u^{(n)}_k) 
=M(u^{(0)},\ldots,u^{(n)})
=M(u)
=p_n.
}
This proves that $(p_n)$ is non-deceasing and, therefore it has a (possibly infinite) limit.
\end{proof}

\section{Applications}

In this section we demonstrate the consequences of our results for Gini means and also for the Gaussian product of 
symmetric, homogeneous, increasing, Jensen concave and repetition invariant means, in particular, the Gaussian product 
of Hölder means.

\subsection{Gini means}

Gini means are symmetric and repetition invariant and min-diminishing (first two properties are simple while the third 
one was proved in \cite{Pal82a}). Moreover, by the results of Losonczi \cite{Los71a,Los71c}, the Gini mean $\G_{p,q}$ 
is increasing and Jensen concave if and only if $pq\le 0$ and $\min(p,q) \le 0 \le \max(p,q) \le 1$, respectively. In 
particular it implies that H\"older mean $\P_p$ is Jensen concave if and only if $p\le1$.

In view of Theorem~\ref{thm:Pas15}, we have the characterization of pairs $(p,q)$ such that $\G_{p,q}$ is a Hardy mean.
In order to calculate the Hardy constant of Gini means using Corollary~\ref{cor:main_hom}, we need to establish 
the following result.

\begin{lem}\label{lem:gini_harmonic}
Let $p,q \in (-\infty,1)$. Then 
\Eq{*}{
  \lim_{n \to \infty} n \cdot \G_{p,q}\left(1,\tfrac 12,\ldots,\tfrac 1n \right) 
  = \begin{cases}
    \Big( \dfrac{1-q}{1-p} \Big)^{\frac1{p-q}} & \mbox{if } p \ne q, \\[4mm]
    e^{\frac1{1-p}} & \mbox{if } p=q.
    \end{cases}
}
\end{lem}

\begin{proof}
For every $s \in (-1,\infty)$, one has
\Eq{*}{
  \lim_{n\to \infty} \frac1n \sum_{i=1}^n \Big( \frac in\Big)^s=\int_0^1 x^s dx = \frac1{1+s}.
}
Using this equality, for $p,\,q <1$, $p \ne q$, we simply obtain
\Eq{*}{
\lim_{n \to \infty} n \cdot \G_{p,q}\left(1,\tfrac 12,\ldots,\tfrac 1n \right) 
&=\lim_{n \to \infty} n \cdot \left( 
\frac{1+2^{-p}+3^{-p}+\cdots+n^{-p}}{1+2^{-q}+3^{-q}+\cdots+n^{-q}}\right)^{\frac1{p-q}} \\
&=\lim_{n \to \infty} \left( \frac
{\frac1n\left[\left(\tfrac 1n\right)^{-p}+\left(\tfrac 2n\right)^{-p}+\left(\tfrac 3n\right)^{-p}
+\cdots+\left(\tfrac {n-1}n\right)^{-p}+1\right]}
{\frac1n\left[\left(\tfrac 1n\right)^{-q}+\left(\tfrac 2n\right)^{-q}+\left(\tfrac 3n\right)^{-q}
+\cdots+\left(\tfrac{n-1}n\right)^{-q}+1\right]}\right)^{\frac1{p-q}}\\
&=\left( \frac{1-q}{1-p} \right)^{\frac1{p-q}}.
}
The proof for the case $p=q<1$ is analogous.
\end{proof}

Using this lemma and the properties that are mentioned just before, we obtain the following

\begin{cor}
Let $p,q \in \R$, $\min(p,q) \le 0 \le \max(p,q)<1$. Then
\Eq{*}{
\Hc_\infty(\G_{p,q})
=\begin{cases} 
\left( \dfrac{1-q}{1-p} \right)^{\frac1{p-q}} & p \ne q, \\
e & p=q=0.
\end{cases}
}
\end{cor}
\begin{proof} Due to the assumption $\min(p,q) \le 0 \le \max(p,q)<1$ and in view of the results of Losonczi
\cite{Los71a,Los71c}, the Gini mean $\G_{p,q}$ is increasing and Jensen concave. Furthermore, 
$\G_{p,q}$ is symmetric, homogeneous, and repetition invariant. Therefore, by Corollary~\ref{cor:main_hom} and
Lemma~\ref{lem:gini_harmonic}, we have
\Eq{*}{
\Hc_\infty(\G_{p,q})=\lim_{n \to \infty} n \cdot \G_{p,q}\left(1,\tfrac 12,\ldots,\tfrac 1n \right)
=\begin{cases} 
\left( \dfrac{1-q}{1-p} \right)^{\frac1{p-q}} & p \ne q, \\
e & p=q=0,
\end{cases}
}
which was to be proved.
\end{proof}

\subsection{Gaussian product}
\begin{prop}
Let $N \in \N$ and let $M_1,\dots,M_N \colon \bigcup_{n=1}^{\infty} \R_+^n \to\R_{+}$ 
be  symmetric, homogeneous, increasing, Jensen concave and repetition invariant means.
If $M_i$ is Hardy for each $i\in\{1,\ldots,N\}$, then so is their Gaussian product $\mgp$ and
\Eq{corgauss}{
\Hc_\infty\left(\mgp\right)= \mgp \big(\Hc_\infty(M_1),\ldots,\Hc_\infty(M_N) \big).}
\end{prop}

\begin{proof}
In view of Lemma~\ref{lem:Gaussian_concave}, the Gaussian product $\mgp$ is a symmetric, homogeneous, increasing, 
Jensen concave and repetition invariant mean. The Jensen concavity and the local boundedness by the 
Bernstein--Doetsch Theorem implies that $\mgp$ is concave and therefore it is also continuous (see \cite{BerDoe15}, 
\cite{Kuc85}). Thus, by Corollary~\ref{cor:main_hom}, we have
\Eq{*}{
\Hc_\infty(\mgp)
&= \lim_{n \to \infty} n \cdot \mgp\left(1,\tfrac 12,\ldots,\tfrac 1n \right) \\
&=\lim_{n \to \infty} n \cdot \mgp\left(M_1(1,\tfrac 12,\ldots,\tfrac 1n),\ldots,M_N(1,\tfrac 12,\ldots,\tfrac 1n) 
\right) \\
&=\lim_{n \to \infty} \mgp\left(nM_1(1,\tfrac 12,\ldots,\tfrac 1n),\ldots,nM_N(1,\tfrac 12,\ldots,\tfrac 1n) \right) \\
&= \mgp\left(\lim_{n \to \infty}nM_1(1,\tfrac 12,\ldots,\tfrac 1n),\ldots,\lim_{n \to \infty}nM_N(1,\tfrac 
12,\ldots,\tfrac 1n) \right) \\
&= \mgp\left(\Hc_\infty(M_1),\ldots,\Hc_\infty(M_N) \right),
}
which proves formula \eq{corgauss}.
\end{proof}

\begin{cor}
Let $N \in \N$ and $(\lambda_1,\dots,\lambda_N) \in \R^N$ then the Gaussian product $\P_{\otimes}$ of the Hölder means  
$\P_{\lambda_1},\dots,\P_{\lambda_N}$ is a Hardy mean if and only if $\max_{1\leq k\leq N} \lambda_k<1$. Furthermore,
in this case, 
\Eq{cor}{
\Hc_\infty\left(\P_{\otimes}\right)
=\P_{\otimes}\big(\Hc_\infty(\P_{\lambda_1}),\ldots,\Hc_\infty(\P_{\lambda_N})\big).
}
\end{cor}

\begin{proof}
The first part of the statement of the above Corollary was proved in \cite{Pas15c} by Pasteczka. If $\lambda_k<1$, then 
$\P_{\lambda_k}$ is a Jensen concave mean, therefore \eq{cor} is a particular case of \eq{corgauss}.
\end{proof}

For example, for the \emph{geometric-harmonic mean} $\P_{-1} \otimes \P_{0}$, i.e., for the Gaussian product of the 
harmonic mean $\P_{-1}$ and the geometric mean $\P_{0}$, we get
\Eq{*}{
\Hc_\infty(\P_{-1} \otimes \P_{0})
=(\P_{-1} \otimes \P_{0})(\Hc_\infty(\P_{-1}),\Hc_\infty(\P_{0}))
=(\P_{-1} \otimes \P_{0})(2,e) \approx 2,318.
}

%\bibliography{publ,funcequ}

\begin{thebibliography}{10}

\bibitem{AczDar63c}
J.~Aczél and Z.~Daróczy.
\newblock {Über verallgemeinerte quasilineare {M}ittelwerte, die mit
  {G}ewichtsfunktionen gebildet sind}.
\newblock {\em Publ. Math. Debrecen}, 10:171–190, 1963.

\bibitem{Baj58}
M.~Bajraktarević.
\newblock {Sur une équation fonctionnelle aux valeurs moyennes}.
\newblock {\em Glasnik Mat.-Fiz. Astronom. Društvo Mat. Fiz. Hrvatske Ser.
  II}, 13:243–248, 1958.

\bibitem{Baj69}
M.~Bajraktarević.
\newblock {Über die {V}ergleichbarkeit der mit {G}ewichtsfunktionen gebildeten
  {M}ittelwerte}.
\newblock {\em Studia Sci. Math. Hungar.}, 4:3–8, 1969.

\bibitem{BajPal09b}
Sz. Baják and Zs. Páles.
\newblock {Computer aided solution of the invariance equation for two-variable
  {G}ini means}.
\newblock {\em Comput. Math. Appl.}, 58:334–340, 2009.

\bibitem{BajPal09a}
Sz. Baják and Zs. Páles.
\newblock {Invariance equation for generalized quasi-arithmetic means}.
\newblock {\em Aequationes Math.}, 77:133–145, 2009.

\bibitem{BajPal10}
Sz. Baják and Zs. Páles.
\newblock {Computer aided solution of the invariance equation for two-variable
  {S}tolarsky means}.
\newblock {\em Appl. Math. Comput.}, 216(11):3219–3227, 2010.

\bibitem{BajPal13}
Sz. Baják and Zs. Páles.
\newblock {Solving invariance equations involving homogeneous means with the
  help of computer}.
\newblock {\em Appl. Math. Comput.}, 219(11):6297–6315, 2013.

\bibitem{BerDoe15}
F.~Bernstein and G.~Doetsch.
\newblock {Zur {T}heorie der konvexen {F}unktionen}.
\newblock {\em Math. Ann.}, 76(4):514–526, 1915.

\bibitem{BorBor87}
J.~M. Borwein and P.~B. Borwein.
\newblock {\em {Pi and the {AGM}}}.
\newblock {Canadian Mathematical Society Series of Monographs and Advanced
  Texts}. John Wiley \& Sons, Inc., New York, 1987.
\newblock A study in analytic number theory and computational complexity, A
  Wiley-Interscience Publication.

\bibitem{Car32}
T.~Carleman.
\newblock {Sur les fonctions quasi-analitiques}.
\newblock {\em Conférences faites au cinquième congrès des mathématiciens
  scandinaves, Helsinki}, page 181–196, 1932.

\bibitem{Dar71b}
Z.~Daróczy.
\newblock {A general inequality for means}.
\newblock {\em Aequationes Math.}, 7(1):16–21, 1971.

\bibitem{Dar72b}
Z.~Daróczy.
\newblock {Über eine {K}lasse von {M}ittelwerten}.
\newblock {\em Publ. Math. Debrecen}, 19:211–217 (1973), 1972.

\bibitem{Dar05a}
Z.~Daróczy.
\newblock {Functional equations involving means and {G}auss compositions of
  means}.
\newblock {\em Nonlinear Anal.}, 63(5-7):e417–e425, 2005.

\bibitem{DarLos70}
Z.~Daróczy and L.~Losonczi.
\newblock {Über den {V}ergleich von {M}ittelwerten}.
\newblock {\em Publ. Math. Debrecen}, 17:289–297 (1971), 1970.

\bibitem{DarPal82}
Z.~Daróczy and Zs. Páles.
\newblock {On comparison of mean values}.
\newblock {\em Publ. Math. Debrecen}, 29(1-2):107–115, 1982.

\bibitem{DarPal83}
Z.~Daróczy and Zs. Páles.
\newblock {Multiplicative mean values and entropies}.
\newblock In {\em {Functions, series, operators, Vol. I, II (Budapest, 1980)}},
  page 343–359. North-Holland, Amsterdam, 1983.

\bibitem{DarPal02c}
Z.~Daróczy and Zs. Páles.
\newblock {Gauss-composition of means and the solution of the
  {M}atkowski–{S}utô problem}.
\newblock {\em Publ. Math. Debrecen}, 61(1-2):157–218, 2002.

\bibitem{DarPal03a}
Z.~Daróczy and Zs. Páles.
\newblock {The {M}atkowski–{S}utô problem for weighted quasi-arithmetic
  means}.
\newblock {\em Acta Math. Hungar.}, 100(3):237–243, 2003.

\bibitem{FosPhi84a}
D.~M.~E. Foster and G.~M. Phillips.
\newblock {The arithmetic-harmonic mean}.
\newblock {\em Math. Comp.}, 42(165):183–191, 1984.

\bibitem{Gau18}
C.~F. Gauss.
\newblock {Nachlass: Aritmetisch-geometrisches Mittel}.
\newblock In {\em {Werke 3 (Göttingem 1876)}}, page 357–402. Königliche
  Gesellschaft der Wissenschaften, 1818.

\bibitem{Gla11b}
D.~Głazowska.
\newblock {A solution of an open problem concerning {L}agrangian mean-type
  mappings}.
\newblock {\em Cent. Eur. J. Math.}, 9(5):1067–1073, 2011.

\bibitem{Gla11a}
D.~Głazowska.
\newblock {Some {C}auchy mean-type mappings for which the geometric mean is
  invariant}.
\newblock {\em J. Math. Anal. Appl.}, 375(2):418–430, 2011.

\bibitem{Har25a}
G.~H. Hardy.
\newblock {Note on a theorem of {H}ilbert concerning series of positive terms}.
\newblock {\em Proc. London Math. Soc.}, 23(2), 1925.

\bibitem{HarLitPol34}
G.~H. Hardy, J.~E. Littlewood, and G.~Pólya.
\newblock {\em {Inequalities}}.
\newblock Cambridge University Press, Cambridge, 1934.
\newblock (first edition), 1952 (second edition).

\bibitem{KalSze27}
T.~Kaluza and G.~Szegő.
\newblock {Über {R}eihen mit lauter positiven {G}liedern}.
\newblock {\em J. London Math. Soc.}, 2:266–272, 1927.

\bibitem{Ked94}
K.~S. Kedlaya.
\newblock {Proof of a mixed arithmetic-mean, geometric-mean inequality}.
\newblock {\em Amer. Math. Monthly}, 101(4):355–357, 1994.

\bibitem{Ked99}
K.~S. Kedlaya.
\newblock {Notes: {A} {W}eighted {M}ixed-{M}ean {I}nequality}.
\newblock {\em Amer. Math. Monthly}, 106(4):355–358, 1999.

\bibitem{Kno28}
K.~Knopp.
\newblock {Über {R}eihen mit positiven {G}liedern}.
\newblock {\em J. London Math. Soc.}, 3:205–211, 1928.

\bibitem{Kuc85}
M.~Kuczma.
\newblock {\em {An {I}ntroduction to the {T}heory of {F}unctional {E}quations
  and {I}nequalities}}, volume 489 of {\em {Prace Naukowe Uniwersytetu
  Śląskiego w Katowicach}}.
\newblock Państwowe Wydawnictwo Naukowe — Uniwersytet Śląski,
  Warszawa–Kraków–Katowice, 1985.
\newblock 2nd edn. (ed. by A. Gilányi), Birkhäuser, Basel, 2009.

\bibitem{KufMalPer07}
A.~Kufner, L.~Maligranda, and L.E. Persson.
\newblock {\em {The Hardy Inequality: About Its History and Some Related
  Results}}.
\newblock Vydavatelsk{\`y} servis, 2007.

\bibitem{KufPer00}
A.~Kufner and L.-E. Persson.
\newblock {\em {Integral {I}nequalities with {W}eights}}.
\newblock Matematický ústav AVČR, Prague, 1990.

\bibitem{Leh71}
D.~H. Lehmer.
\newblock {{O}n the compounding of certain means}.
\newblock {\em J. Math. Anal. Appl.}, 36:183–200, 1971.

\bibitem{Los70a}
L.~Losonczi.
\newblock {Über den {V}ergleich von {M}ittelwerten die mit
  {G}ewichtsfunktionen gebildet sind}.
\newblock {\em Publ. Math. Debrecen}, 17:203–208 (1971), 1970.

\bibitem{Los71a}
L.~Losonczi.
\newblock {Subadditive {M}ittelwerte}.
\newblock {\em Arch. Math. (Basel)}, 22:168–174, 1971.

\bibitem{Los71c}
L.~Losonczi.
\newblock {Subhomogene {M}ittelwerte}.
\newblock {\em Acta Math. Acad. Sci. Hungar.}, 22:187–195, 1971.

\bibitem{Los71b}
L.~Losonczi.
\newblock {Über eine neue {K}lasse von {M}ittelwerten}.
\newblock {\em Acta Sci. Math. (Szeged)}, 32:71–81, 1971.

\bibitem{Los73a}
L.~Losonczi.
\newblock {General inequalities for nonsymmetric means}.
\newblock {\em Aequationes Math.}, 9:221–235, 1973.

\bibitem{Los77}
L.~Losonczi.
\newblock {Inequalities for integral mean values}.
\newblock {\em J. Math. Anal. Appl.}, 61(3):586–606, 1977.

\bibitem{Mat99b}
J.~Matkowski.
\newblock {Iterations of mean-type mappings and invariant means}.
\newblock {\em Ann. Math. Sil.}, (13):211–226, 1999.
\newblock European Conference on Iteration Theory (Muszyna-Złockie, 1998).

\bibitem{Mat02b}
J.~Matkowski.
\newblock {On iteration semigroups of mean-type mappings and invariant means}.
\newblock {\em Aequationes Math.}, 64(3):297–303, 2002.

\bibitem{Mat05}
J.~Matkowski.
\newblock {Lagrangian mean-type mappings for which the arithmetic mean is
  invariant}.
\newblock {\em J. Math. Anal. Appl.}, 309(1):15–24, 2005.

\bibitem{Mat13}
J.~Matkowski.
\newblock {Iterations of the mean-type mappings and uniqueness of invariant
  means}.
\newblock {\em Annales Univ. Sci. Budapest., Sect. Comp.}, 41:145–158, 2013.

\bibitem{MatPal15}
J.~Matkowski and Zs. Páles.
\newblock {Characterization of generalized quasi-arithmetic means}.
\newblock {\em Acta Sci. Math. (Szeged)}, 81(3–4):447–456, 2015.

\bibitem{MitPecFin91}
D.~S. Mitrinović, J.~E. Pečarić, and A.~M. Fink.
\newblock {\em {Inequalities {I}nvolving {F}unctions and {T}heir {I}ntegrals
  and {D}erivatives}}, volume~53 of {\em {Mathematics and its Applications
  (East European Series)}}.
\newblock Kluwer Academic Publishers Group, Dordrecht, 1991.

\bibitem{Mul32}
P.~Mulholland.
\newblock {On the generalization of {H}ardy's inequality}.
\newblock {\em J. London Math. Soc.}, 7:208–214, 1932.

\bibitem{OpiKuf90}
B.~Opic and A.~Kufner.
\newblock {\em {Hardy-type {I}nequalities}}, volume 219 of {\em {Pitman
  Research Notes in Mathematics}}.
\newblock Longman Scientific \& Technical, Harlow, 1990.

\bibitem{Pas15c}
P.~Pasteczka.
\newblock {On negative results concerning {H}ardy means}.
\newblock {\em Acta Math. Hungar.}, 146(1):98–106, 2015.

\bibitem{Pas15a}
P.~Pasteczka.
\newblock {Scales of quasi-arithmetic means determined by an invariance
  property}.
\newblock {\em J. Difference Equ. Appl.}, 21(8):742–755, 2015.

\bibitem{PecSto01}
J.~E. Pečarić and K.~B. Stolarsky.
\newblock {Carleman's inequality: history and new generalizations}.
\newblock {\em Aequationes Math.}, 61(1–2):49–62, 2001.

\bibitem{Pal82b}
Zs. Páles.
\newblock {A generalization of the {M}inkowski inequality}.
\newblock {\em J. Math. Anal. Appl.}, 90(2):456–462, 1982.

\bibitem{Pal82a}
Zs. Páles.
\newblock {Characterization of quasideviation means}.
\newblock {\em Acta Math. Acad. Sci. Hungar.}, 40(3-4):243–260, 1982.

\bibitem{Pal83a}
Zs. Páles.
\newblock {Inequalities for homogeneous means depending on two parameters}.
\newblock In E.~F. Beckenbach and W.~Walter, editors, {\em {General
  Inequalities, 3 (Oberwolfach, 1981)}}, volume~64 of {\em {International
  Series of Numerical Mathematics}}, page 107–122. Birkhäuser, Basel, 1983.

\bibitem{Pal83b}
Zs. Páles.
\newblock {On complementary inequalities}.
\newblock {\em Publ. Math. Debrecen}, 30(1-2):75–88, 1983.

\bibitem{Pal83c}
Zs. Páles.
\newblock {On {H}ölder-type inequalities}.
\newblock {\em J. Math. Anal. Appl.}, 95(2):457–466, 1983.

\bibitem{Pal84a}
Zs. Páles.
\newblock {Inequalities for comparison of means}.
\newblock In W.~Walter, editor, {\em {General Inequalities, 4 (Oberwolfach,
  1983)}}, volume~71 of {\em {International Series of Numerical Mathematics}},
  page 59–73. Birkhäuser, Basel, 1984.

\bibitem{Pal85a}
Zs. Páles.
\newblock {Ingham {J}essen's inequality for deviation means}.
\newblock {\em Acta Sci. Math. (Szeged)}, 49(1-4):131–142, 1985.

\bibitem{Pal87d}
Zs. Páles.
\newblock {On the characterization of quasi-arithmetic means with weight
  function}.
\newblock {\em Aequationes Math.}, 32(2-3):171–194, 1987.

\bibitem{Pal88a}
Zs. Páles.
\newblock {General inequalities for quasideviation means}.
\newblock {\em Aequationes Math.}, 36(1):32–56, 1988.

\bibitem{Pal88d}
Zs. Páles.
\newblock {On a {P}exider-type functional equation for quasideviation means}.
\newblock {\em Acta Math. Hungar.}, 51(1-2):205–224, 1988.

\bibitem{Pal88e}
Zs. Páles.
\newblock {On homogeneous quasideviation means}.
\newblock {\em Aequationes Math.}, 36(2-3):132–152, 1988.

\bibitem{Pal00a}
Zs. Páles.
\newblock {Nonconvex functions and separation by power means}.
\newblock {\em Math. Inequal. Appl.}, 3(2):169–176, 2000.

\bibitem{PalPer04}
Zs. Páles and L.-E. Persson.
\newblock {Hardy type inequalities for means}.
\newblock {\em Bull. Austr. Math. Soc.}, 70(3):521–528, 2004.

\bibitem{Sch82}
I.~J. Schoenberg.
\newblock {\em {Mathematical time exposures}}.
\newblock Mathematical Association of America, Washington, DC, 1982.

\bibitem{ToaToa05}
G.~Toader and S.~Toader.
\newblock {\em {Greek means and the arithmetic-geometric mean}}.
\newblock {RGMIA Monographs}. Victoria University, 2005.

\end{thebibliography}
%\bibliographystyle{plain}

\def\cprime{$'$}

\end{document}